\documentclass[letterpaper, 12 pt, twoside]{amsart}
\usepackage{bwieneck}[2011/03/03]
\usepackage[paper = letterpaper, left = 3.4cm, right = 3.4cm, headsep = 6mm,
footskip = 10mm, top = 34mm, bottom = 34mm, footnotesep=5mm, headheight =
2cm]{geometry}

\usepackage{color}
\usepackage[T1]{fontenc}
\usepackage{lmodern} %saubere schrift
\usepackage{mathrsfs}
\usepackage{textcomp}
\usepackage{amssymb}
\usepackage{mathtools} %fuer \coloneqq

\usepackage{tikz}
\normalsize

\usepackage[hypertexnames=false,backref=page,pdftex,
	breaklinks=true,
	extension=pdf,
	colorlinks=true,
	linkcolor=blue,
	citecolor=blue,
	urlcolor=blue,
	pdfpagemode=NonOutlines,
	]{hyperref}
\renewcommand{\epsilon}{\varepsilon}
\title[]{On polarization types of Lagrangian fibrations}

\author{Benjamin Wieneck}
\address{Benjamin Wieneck\\ Institut für Algebraische Geometrie \\Leibniz Universit\"at Hannover\\Welfengarten 1\\30167 Hannover\\Germany}
\email{wieneck@math.uni-hannover.de}

\begin{document}

\thanks{Mathematics Subject Classification 32J27, 14D06, 32Q15, 53D12, 53C26, 32G13, 14D20}
\begin{abstract} The generic fiber of a Lagrangian fibration on an irreducible holomorphic symplectic manifold
is an abelian variety. Associate a polarization type to such Lagrangian fibrations coming from
polarizations on a generic fiber. We prove that this polarization type is constant in families of
Lagrangian fibrations. Further, we determine the polarization type of $\kdrei$--type fibrations and
conjecture that the polarization type should only depend on the deformation type of the total space. 
\end{abstract}

\maketitle
\tableofcontents
\section{Introduction} \label{intro}

The geometry of irreducible holomorphic symplectic manifolds or compact hyperkähler manifolds seems to be quite rigid since very few deformation types are known. The only possible nontrivial fibrations such manifolds can admit are Lagrangian fibrations as D. Matsushita showed, see Theorem \ref{mats}. Lagrangian fibrations help us to understand the geometry of irreducible holomorphic symplectic manifolds. It is hoped that Lagrangian fibrations will be useful for the classification of irreducible holomorphic symplectic manifolds, see \cite{sawonab}.
%Many more authors have studied Lagrangian fibrations, see for instance 
 \\ 

Let $f : X \rightarrow B$ be a Lagrangian fibration. It is well known that all smooth fibers are abelian varieties even if $X$ is not projective. Given a smooth fiber $F$ an immediate question is to ask for polarizations on it which is by definition the first Chern class $H = c_1(L)$ of an ample line bundle $L$ of $F$. We want to consider types of such polarizations on $F$.

The classical notion of polarizations on abelian varieties and their types can be found in the book \cite{BL} of H. Lange and C. Birkenhake, but we also recall the notion in section \ref{poltypes}.

By an observation of C. Voisin, see Proposition \ref{camp}, it is known that for each smooth fiber $F$ one can find a Kähler class $\omega$ on $X$ such that the restriction $\omega|_F$ is integral and primitive and hence defines a polarization on $F$. An ad--hoc definition of the polarization type of a Lagrangian fibration would be to set $\pol(f) \coloneqq \pol(\omega|_F)$, where the latter is the polarization type of the polarization on $F$ given by $\omega|_F$. Indeed, this does not depend on the chosen smooth fiber $F$ and the chosen $\omega$, see Proposition \ref{polcon} and Proposition \ref{specialpol}. 

The proof of the main result Theorem \ref{maini} below, involves moduli theory of \emph{$\kdrei$--type fibrations} \ie Lagrangian fibrations $f \colon X \rightarrow \IP^n$ such that $X$ is of \emph{$\kdrei$--type}.  %, this means $X$ is deformation equivalent to the Hilbert scheme $S^{[n]}$ of $n$ points of some K$3$ surface $S$. \\

\begin{theo} \label{maini} Let $f \colon X \rightarrow B$ be a Lagrangian fibration with $\dim X = 2n$. Then we associate to $f$ a tuple $\pol(f) \in \IZ^n$ of positive integers called the \emph{polarization type} such that the following holds. \begin{enumerate}
\item \emph{Theorem \ref{samedef}} The polarization type is a deformation invariant of the fibration \ie if $f' \colon X' \rightarrow B'$ is a Lagrangian fibration deformation equivalent to $f$\footnote{This means we have an $S$--morphism $\phi \colon \kx \rightarrow P$ such that $S$ is a connected complex space with finitely many irreducible components, $\kX \rightarrow S$ is a family of irreducible holomorphic symplectic manifolds and $P \rightarrow S$ is a family of projective varieties such that $\phi_t \coloneqq \phi|_{\kX_t} \colon \kX_t \rightarrow P_t$ is a Lagrangian fibration for all $t \in S$ and there are points $t_i \in S$, $i = 1,2$, such that $f = \phi_{t_1}$ and $f' = \phi_{t_2}$.} then $\pol(f) = \pol(f')$.
\item \emph{Theorem \ref{main}} If $X$ is of $\kdrei$--type, then $\pol(f)$ is principal \ie $\pol(f)= (1, \ldots, 1)$.
\end{enumerate}
\end{theo} 

Ultimately we expect the polarization type $\pol(f)$ only to depend on the deformation type of the total space $X$ of $f \colon X \rightarrow B$. But this is possibly a too optimistic conjecture. \\

In section \ref{modulioflagrangian} the moduli theory of Lagrangian fibrations is explained which is mostly a recollection of known facts for the convenience of the reader. In the case of $\kdrei$--type fibrations this relies on methods developed by E. Markman in \cite{eyal1} and \cite{eyal2}. Besides that two results of D. Matsushita \cite{Mat09}, \cite{matiso} play an important role. The former one, see Theorem \ref{matsh}, states that every Lagrangian fibration can be considered as a member of a family of Lagrangian fibrations parametrized by a small representative of deformation space $\Def(X,L)$ of the pair $(X,L)$, where $L$ is the pullback of an ample line bundle on the base space. Using this theory we describe how to obtain a connected component of the moduli space of $\kdrei$--type fibrations. 

In section \ref{poltypes} we construct the polarization of a Lagrangian fibration, prove the basic properties and the first part of Theorem \ref{maini}. In particular a relation between the geometry of the moduli of Lagrangian fibrations and the polarization type is given, compare for instance Propostion \ref{defo} and Theorem \ref{same}. \\

Section \ref{beauvillemukai} recalls the notion of Beauville--Mukai systems which are examples of $\kdrei$--type fibrations. Further it is shown in Corollarly \ref{prin} that their polarization type is principal \ie given by $(1, \ldots, 1) \in \IZ^n$ which follows from a work from C. Ciliberto and G. van der Geer \cite{cili} which is needed for the proof of Theorem $\ref{main}$ (ii). \\

In section \ref{kdrei} we prove Theorem \ref{maini} (ii) \ie that the polarization types of $\kdrei$--fibrations is principal. The rough idea of the proof is the following. Every connected component of the moduli of $\kdrei$-type fibrations contains a Beauville--Mukai system which was found in \cite{eyal2}. Then each two fibrations in such a component are deformation equivalent as fibrations hence the polarization type must be principal by Theorem \ref{maini} (i) since the polarization types of Beauville--Mukai systems are principal. A key tool for seeing if two Lagrangian fibrations belong to the same connected component of the moduli is a monodromy invariant function, see Lemma \ref{invariant} which was introduced in \cite{eyal2}. \\

\textbf{Acknowledgements:} I thank my advisor Klaus Hulek, Eyal Markman, Christian Lehn and Malek Joumaah for helpful discussions. In particular Christian Lehn for also introducing me to this problem. Especially I want to thank Eyal Markman for hospitality at the University of Massachusetts in Amherst and explaining me so much about the moduli of Lagrangian fibrations. I also thank the referee, who read this work carefully and pointed out several mistakes and improvements. I gratefully acknowledge the support of the DFG Graduiertenkolleg
1463 ``Analysis, Geometry and String Theory''.

\section{Irreducible symplectic manifolds and their fibrations}

In this section we recall the basic facts about irreducible holomorphic symplectic manifolds and their fibrations.

\begin{defin}A compact Kähler manifold $X$ is called \emph{hyperkähler} or \emph{irreducible holomorphic symplectic} if $X$ is simply connected and $H^0(X, \Omega_X^2)$ is generated by a nowhere degenerate holomorphic two--form $\sigma$. \end{defin} Note that $\sigma$ is automatically symplectic since every holomorphic form on a compact Kähler manifold is closed.

The irreducible holomorphic symplectic manifolds of dimension two are precisely the K$3$ surfaces. The most basic higher dimensional example is provided by the Douady space $S^{[n]}$ of $n$ points for a K$3$ surface $S$ which parametrizes zero--dimensional subspaces of $S$ of length $l(Z) \coloneqq \dim_{\CC} \ko_Z(Z) = n$. A. Beauville \cite{Bea84} showed that $S^{[n]}$ is an irreducible holomorphic symplectic manifold of dimension $2n$. An irreducible holomorphic symplectic manifold is called of \emph{$\kdrei$--type} or \emph{$\kdrei$--type manifold} if it is deformation equivalent to $S^{[n]}$ for a K$3$ surface $S$. \\

The second cohomology $H^2(X,\IZ)$ of any irreducible holomorphic symplectic manifold $X$ admits the well known \emph{Beauville--Bogomolov--Fujiki} quadratic form $q_X$ which is non--degenerate and of signature $(3, b_2(X) - 3)$, see \cite[23.3]{huy1}. The associated bilinear form is denoted by $(\cdot, \cdot)$. On an abstract lattice we also denote the bilinear form by $(\cdot, \cdot)$. The lattice $H^2(X,\IZ)$ with the Beauville--Bogomolov--Fujiki form is invariant under deformations of the manifold $X$. For manifolds of $\kdrei$--type, $n \neq 1$, this lattice is isometric to the \emph{$\kdrei$--type lattice}
$$ E_8(-1)^{\oplus 2} \oplus U^{\oplus 3} \oplus \left\langle 2-2n \right\rangle \, ,$$
see \cite[Prop. 6]{Bea84} where $\left\langle 2-2n\right\rangle$ denotes the lattice of rank one with generator $l$ such that $(l,l) = 2-2n$, $E_8(-1)$ the negative definite root lattice of type $E_8$ and $U$ the unimodular rank two hyperbolic lattice.

If $\Lambda$ is a lattice isometric to the second cohomology of an irreducible holomorphic symplectic manifold $X$, then  a \emph{marking} on $X$ is the choice of an isometry $\eta : H^2(X,\IZ) \rightarrow \Lambda$. The tuple $(X,\eta)$ is then called a \emph{marked pair} or a \emph{marked irreducible holomorphic symplectic manifold}. \\

If $X$ is a fixed irreducible holomorphic symplectic manifold set $\Lambda \coloneqq H^2(X,\IZ)$ and consider the Kuranishi family $\pi : \gX \rightarrow \Def(X)$ with $\gX_0 \coloneqq \pi^{-1}(0) = X$. We will view the base $\Def(X)$ sometimes as a germ but also as a representative which we choose small enough \ie simply connected and that all fibers are irreducible holomorphic symplectic. Then by Ehresmann's theorem we can choose a trivialization $\Sigma : R^2 \pi_{\star}\IZ \rightarrow \Lambda$ also called a \emph{marking}. Then define the \emph{local period map} by
$$ \kp : \Def(X) \longrightarrow \IP(\Lambda_{\CC}) \, , \ \ t \longmapsto [\Sigma_t(H^{2,0}(\gX_t))]$$
where $\Lambda_{\CC} \coloneqq \Lambda \otimes \CC$ and $\Sigma_t : (R^2 \pi_{\star}\IZ)_t \otimes_{\IZ} \CC \rightarrow \Lambda_{\CC}$ is the induced map of stalks. It takes values in the \emph{period domain of type $\Lambda$} \cite[22.3, 25.2]{huy1} namely
$$\Omega_{\Lambda} \ \coloneqq \ \left\{ p \in \IP(\Lambda_{\CC}) \ | \ (p,p) = 0 \text{ and } (p, \bar{p}) > 0 \right\} $$
which is connected since the signature of $q_X$ is $(3, \rk \Lambda - 3)$, see \cite[Thm. 5]{Bea84}.
\begin{theo}[\textsc{Local Torelli},  \cite{Bea84}, Thm. 5] If $\Def(X)$ is chosen small enough, the period map $\kp : \Def(X) \rightarrow \Omega_{\Lambda}$ is an open embedding. \end{theo}
Two marked pairs $(X_i, \eta_i)$, $i=1,2$, are called \emph{isomorphic} if there is an isomorphism $f : X_1 \rightarrow X_2$ such that $\eta_2 = \eta_1 \circ f^{\star}$.
There exists a \emph{moduli space of marked pairs} $\gM_{\Lambda} \coloneqq \left\{ (X,\eta) \text{ marked pair } \right\}/\cong$ which can be constructed by glueing all deformation spaces $\Def(X)$ of irreducible holomorphic symplectic manifolds $X$ with $H^2(X,\IZ)$ isometric to $\Lambda$. This gives a non--Hausdorff complex manifold of dimension $\rk \Lambda - 2$. The \emph{global period map} is  
$$ \kp \ : \ \gM_{\Lambda} \longrightarrow \Omega_{\Lambda} \, , \ \ \ (X,\eta) \longmapsto [\eta(H^{2,0}(X))] $$ 
%is locally given by $\kp : \Def(X) \rightarrow \Omega_{\Lambda}$ and hence
and is a local biholomorphism by the Local Torelli. If one takes an arbitrary connected component $\gM^{\circ}_{\Lambda}$ of $\gM_{\Lambda}$ then by a result of D. Huybrechts \cite[Prop. 25.12]{huy1} the restriction $\kp : \gM^{\circ}_{\Lambda} \rightarrow \Omega_{\Lambda}$ is surjective. 

If $L$ denotes a line bundle on $X$ by abuse of notation we also denote the universal family of the pair $(X,L)$ by $\pi : \gX_L \rightarrow \Def(X,L)$ which comes with a universal line bundle $\kl$ on $\gX_L$ such that $(\gX_L)_0 = X$ and $\kl_0 = L$, see \cite[Cor. 1]{Bea84}, where $\kl_t := \kl|_{\gX_t}$ and $0 \in \Def(X)$ is a reference point. We consider again $\Def(X,L)$ as a germ but as well as a small open representative. A representative of $\Def(X,L)$ is locally given by $(c_1(L), \cdot) = 0$ in $\Omega_{\Lambda}$ hence it is a smooth hypersurface in $\Def(X)$, see \cite[26.1]{huy1} and one defines $\gX_L$ as the preimage of it under $\pi$. The family $\pi : \gX_L \rightarrow \Def(X,L)$ is the restriction of the Kuranishi family $\pi : \gX \rightarrow \Def(X)$ to $\gX_L$ and $\Def(X,L)$. \\

For completeness we give the following definitions.
\begin{defin} Let $X_i$, $i = 1,2$, be two irreducible holomorphic symplectic manifolds. An isomorphism $P : H^2(X_1, \IZ) \rightarrow H^2(X_2, \IZ)$ is called a \emph{parallel transport operator} if there exists a family $\pi : \kx \rightarrow S$ of irreducible holomorphic symplectic manifolds, points $t_i$ such that $\kx_{t_i} = X_i$ and a continuous path $\gamma$ such that the parallel transport $P_{\gamma}$ along $\gamma$ in the local system $R^2 \pi_{\star} \IZ$ coincides with $P$. For $X \coloneqq X_1 = X_2$ it is also called a \emph{monodromy operator} and the subgroup $\Mon^2(X)$ of $O(H^2(X,\IZ))$ generated by monodromy operators is called the \emph{monodromy group}. \end{defin}

Due to D. Matsushita much is known about nontrivial fiber structures on irreducible holomorphic symplectic manifolds.
\begin{theo}[\textsc{Matsushita}, \cite{Mat99}, \cite{Mat00}, \cite{Mat01}, \cite{Mat03}]\label{mats} Let $f : X \rightarrow B$ be a surjective holomorphic map with connected fibers from an irreducible holomorphic symplectic manifold $X$ of dimension $2n$ to a normal complex space $B$ such that $0 < \dim B < 2n$. Then the following statements hold.
\begin{enumerate} 
\item $B$ is projective of dimension $n$ and its Picard number is $\rho(B) = 1$.
\item For all $t \in B$ the fiber $X_t \coloneqq f^{-1}(t)$ is Lagrangian subspace \ie $\sigma|_{X^{\text{\emph{reg}}}_t} = 0$ where $X^{\text{\emph{reg}}}_t$ denotes the smooth part of $X_t$.
	\item If $X_t$ is smooth then it is a projective complex torus \ie an abelian variety. \end{enumerate}
	\end{theo} 
Such a fibration $f : X \rightarrow B$ as in the Theorem is called a \emph{Lagrangian fibration}. If $X$ is a $\kdrei$--type manifold then we call $f : X \rightarrow B$ a \emph{$\kdrei$--type fibration}. \\

If the base of the Lagrangian fibration is smooth even more is known due to a deep result of J.-M. Hwang which was recently slighty generalized by C. Lehn and D. Greb to the non--projective case.
\begin{theo}[\textsc{Hwang}, \cite{Hwa08}, \cite{LeGr13}] Let $f : X \rightarrow B$ be a Lagrangian fibration such that $B$ is smooth and $\dim X = 2n$. Then $B \cong \IP^n$. \end{theo} If $f : X \rightarrow B$ is a $\kdrei$--type fibration then E. Markman \cite[Thm. 1.3, Rem. 1.8]{eyal1} in combination with a result of D. Matsushita \cite[Thm. 1.2, Cor. 1.1]{matiso} has shown that $B \cong \IP^n$ without assuming smoothness of $B$. By \cite[Appendix]{yoshibase} also in combination with \cite[Thm. 1.2, Cor. 1.1]{matiso} this holds for Lagrangian fibrations on irreducible holomorphic symplectic manifolds deformation equivalent to generalized Kummer manifolds. Generalized Kummer manifolds are another class of irreducible holomorphic symplectic manifolds introduced by A. Beauville \cite{Bea84}. The general conjecture is that the base is always the projective space. \\

The basic example of a Lagrangian fibration is provided by the Douady space of an elliptic K$3$ surface $f : S \rightarrow \IP^1$.  Then one uses the \emph{Douady--Barlet map}
$$ \rho \ : \ S^{[n]} \longrightarrow S^{(n)}\, , \ \ \ Z \longmapsto \sum\limits_{z \in Z} (\dim_{\CC} \ko_{Z,z}) z $$
which is a resolution of singularities of the \emph{$n$--th symmetric product} $S^{(n)} \coloneqq (S \times \cdots \times S)/\Sigma_n$ to obtain a morphism 
$$ S^{[n]} \stackrel{\rho}{\longrightarrow} S^{(n)} \stackrel{f \times \cdots \times f}{\longrightarrow} (\IP^1)^{(n)} \cong \IP^n \,$$
which is a Lagrangian fibration and the smooth fibers are given by products of elliptic curves which are the smooth fibers of $f : S \rightarrow \IP^1$. Note that two--dimensional Lagrangian fibrations are exactly the elliptic K$3$ surfaces.

\section{Moduli space of Lagrangian fibrations}\label{modulioflagrangian}

This section has the purpose to explain what we mean by the \emph{moduli space of Lagrangian fibrations}. For the $\kdrei$--type we describe connected components of it. Many of the constructions and explanations can be found in \cite{eyal1}, \cite{eyal2} and \cite{matiso}. 

D. Matsushita \cite{Mat09} constructed a \emph{local moduli space} for arbitrary Lagrangian fibrations.

\begin{defin}\label{deflf} \begin{enumerate} \item A \emph{family of Lagrangian fibrations} over a connected complex space $S$ with finitely many irreducible components is an $S$--morphism 
$$   \xymatrix{  \kX \ar[rr]^{\phi} \ar[dr] & & P \ar[dl] \\ 
										&	S & 		   } $$
where $\kX \rightarrow S$ is a family of irreducible holomorphic symplectic manifolds and $P \rightarrow S$ is a family of projective varieties such that for every $s \in S$ the restriction $\phi|_{\kX_s} : \kX_s \rightarrow P_s$ to the irreducible homorphic symplectic manifold $\kX_s$ is a Lagrangian fibration. 
\item Two Lagrangian fibrations $f_1$ and $f_2$ are \emph{deformation equivalent} if there is a family of Lagrangian fibrations over a connected complex space $S$ containing $f_1$ and $f_2$.
\end{enumerate}
\end{defin}

Let $\pi : \gX \rightarrow \Def(X)$ denote the Kuranishi family of an irreducible holomorphic symplectic manifold $X = \pi^{-1}(0)$. For a line bundle $L$ on $X$ let $\pi : \gX_L \rightarrow \Def(X,L)$ denote the universal family of the pair $(X,L)$. Further denote by $\kl$ the universal line bundle on $\gX_L$ and set $\kl_t \coloneqq \kl|_{\gX_{L,t}}$. 
\begin{theo}{\cite[Cor. 1.3]{Mat09}}\label{matsh} Let $f : X \rightarrow B$ be a Lagrangian fibration and $L$ be the pullback of a very ample line bundle on $B$. Then $\kl$ is a $\pi$--relatively base point free line bundle \ie after shrinking the representative $\Def(X,L)$ there exists a family of Lagrangian fibrations 
$$   \begin{xy}\xymatrix{  \gX_L \ar[rr]^{\zeta} \ar[dr]_{\pi} & & \IP(\pi_{\star} \kl) \ar[dl] \\ 
										&	\Def(X,L) & 		}   \end{xy}$$
over $\Def(X,L)$ such that $\zeta_0 = f$. 
\end{theo}

Note that $\dim \Def(X) = h^1(X,\kt_X) = b_2(X) - 2$. For a fixed deformation type of irreducible holomorphic symplectic manifolds with lattice $\Lambda$ one can glue all total spaces $\Def(X,L)$ of such families $\gX_L \rightarrow \Def(X,L)$ for $f : X \rightarrow \IP^n$ a Lagrangian fibration and $L$ a line bundle on $X$ as in the theorem. The result is a possibly non--Hausdorff moduli space of Lagrangian fibrations of dimension $b_2(X) - 3$. It is a smooth submanifold of codimension one of the moduli space of marked irreducible holomorphic symplectic manifolds $\gM_{\Lambda}$. % $\gH_{\Lambda}$

In the $\kdrei$--type case E. Markman \cite{eyal2} provides methods to describe a connected component of the moduli space of Lagrangian fibrations lattice--theoretically. \\

From now on in this section, $X$ usually denotes a $\kdrei$--type manifold and $\Lambda$ denotes the $\kdrei$--type lattice. However, some of the statements also holds for arbitrary irreducible holomorphic symplectic manifolds (which will be mentioned). \\

For a period $p \in \Omega_{\Lambda}$ let $\Lambda(p)$ denote the integral Hodge structure of weight two of $\Lambda$ determined by the period $p$ that is
$$\Lambda^{2,0}(p) = p \, , \ \ \ \Lambda^{0,2}(p) = \bar{p} \ \ \text{ and } \ \ \Lambda^{1,1}(p) = \left\{ m \in \Lambda_{\CC} \ | \ (m,p) = (m, \bar{p}) = 0\right\} \, .$$
We also set $\Lambda^{1,1}(p, R) \coloneqq \left\{ m \in \Lambda_R \ | \ (m,p) = 0 \right\}$ for $R \in \left\{ \IZ , \IR \right\}$. Further consider the sets
$$ \kc_p \coloneqq \left\{ m \in \Lambda^{1,1}(p, \IR) \ | \ (m,m) > 0 \right\} \ \text{ and } \ \tilde{\kc}_{\Lambda} \coloneqq \left\{ m \in \Lambda_{\IR} \ | \ (m,m) > 0 \right\} \, .$$ Both sets are sometimes called \emph{positive cone}. The latter one $\tilde{\kc}_{\Lambda}$ is connected and $H^2(\tilde{\kc}_{\Lambda}, \IZ) \cong \IZ$, see \cite[Lem. 4.1]{eyal1}. 
\begin{lem}{\cite[4.3]{eyal2}}\label{orientation}  \begin{enumerate} \item The cone $\kc_p$ has two connected components. 
\item A choice of an orientation of $\tilde{\kc}_{\Lambda}$ determines a connected component $\kc^+_p$ of $\kc_p$ for all $p \in \Omega_{\Lambda}$. \end{enumerate}\end{lem}
Choose a primitive isotropic class $\lambda \in \Lambda$ and set
$$\Omega_{\lambda^\bot} \ \coloneqq \ \left\{ p \in \Omega_{\Lambda} \ | \ (p, \lambda) = 0 \right\} \, .$$
Further choose a generator of $H^2(\tilde{\kc}_{\Lambda},\IZ)$ \ie an orientation of $\tilde{\kc}_{\Lambda}$, which determines a connected component  $\kc^+_p$ for every period $p$ by Lemma \ref{orientation}. For $p \in \Omega_{\lambda^\bot}$, $\lambda$ belongs to $\Lambda^{1,1}(p,\IR)$ and is contained in the boundary of one of the connected components of $\kc_p$ since $\lambda$ is isotropic. Then consider only such periods $p$ for which $\lambda$ belongs to the closure of the distinguished connected component $\kc^+_p$ of $\kc_p$  of Lemma \ref{orientation} \ie
$$\Omega^+_{\lambda^\bot} \ \coloneqq \ \left\{ p \in \Omega_{\lambda^\bot} \ | \ \lambda \in \partial \kc^+_p \right\} \, . $$
%which is a connected component of $\Omega_{\lambda^\bot}$ by construction.

Let $\gM_{\Lambda}$ denote the moduli space of isomorphism classes of marked pairs $(X, \eta)$ of $\kdrei$--type \ie $X$ is an irreducible holomorphic symplectic manifold of $\kdrei$--type and $\eta : H^2(X,\IZ) \rightarrow \Lambda$ is a marking. Choose a connected component $\gM^{\circ}_{\Lambda}$ of $\gM_{\Lambda}$ and consider the period map 
$$ \kp \ : \ \gM^{\circ}_{\Lambda} \longrightarrow \Omega_{\Lambda} \, , \ \ \ (X,\eta) \longmapsto [\eta(H^{2,0}(X))] \, . $$ 

Choose the orientation of $\tilde{\kc}_{\Lambda}$ compatible to $\gM^{\circ}_{\Lambda}$ in the following way. If $(X,\eta) \in \gM^{\circ}_{\Lambda}$ then there is a canonical choice for the connected component of $$\left\{ x \in H^{1,1}(X,\IR) \ | \ (x,x) > 0 \right\}$$ namely the \emph{positive cone} $\kc_X^+$ which contains the Kähler cone $\kk_X$ of $X$. Set $p \coloneqq \kp(X,\eta)$ then $\eta(H^{1,1}(X,\IR)) = \Lambda^{1,1}(p,\IR)$ and compatibility means that $\eta(\kc_X^+) = \kc^+_p$. 

Let $\lambda \in \Lambda$ isotropic. Then for a $(X,\eta) \in \gM^\circ_{\lambda^\bot}$ either $\eta^{-1}(\lambda)$ or $\eta^{-1}(-\lambda)$ belongs to $\partial \kc^+_X$. Assume that the former is the case, otherwise take $-\lambda$. Then set
$$ \gM^{\circ}_{\lambda^\bot} \ \coloneqq \ \kp^{-1}\left( \Omega^+_{\lambda^\bot} \right) \ = \ \left\{ (X,\eta) \in \gM^{\circ}_{\Lambda} \ | \ \eta^{-1}(\lambda) \text{ is of type } (1,1) \text{ and in  } \partial \kc^+_X \ \right\} $$
which is  a connected hypersurface of $\gM^{\circ}_{\Lambda}$ by \cite[Lem. 4.4]{eyal2} (compare with \cite[Cor. 5.11]{eyalprime}). Set 
$$ \gU^{\circ}_{\lambda^\bot} \ \coloneqq \ \left\{ (X,\eta) \in \gM^{\circ}_{\lambda^\bot} \ | \ \eta^{-1}(\lambda) \text{ is nef}  \right\} \, . $$
We claim that this space $\gU^{\circ}_{\lambda^\bot}$ is a connected component of the moduli space of $\kdrei$--type fibrations, see Theorem \ref{modulil} below. Furthermore it is connected and open in $\gM^{\circ}_{\lambda^\bot}$. % $\gH_{\Lambda}$

\begin{rem} Recall that a holomorphic line bundle $L \in \Pic(X) \cong \NS(X)$ on an irreducible holomorphic symplectic manifold $X$ is called \emph{nef} if $c_1(L)$ belongs to the closure of the Kähler cone $\kk_X$ in $H^{1,1}(X,\IR)$.  \end{rem}

The following Lemma should be well known, however we recall it for the reader's convenience. 

\begin{lem} \label{liso}Let $f : X \rightarrow B$ be a Lagrangian fibration with $X$ not necessarily of $\kdrei$--type and let $L \coloneqq f^{\star}A$ be the pullback of a line bundle $A$ on $B$. \begin{enumerate}
\item $L$ is isotropic with respect to the Beauville--Bogomolov quadratic form. 
\item If $A$ admits nontrivial sections then $L$ is nef. If $X$ is of $\kdrei$--type and $A = \ko_{\IP^n}(1)$ on $B = \IP^n$, then $L$ is primitive.
%\item If $A = \ko_{\IP^n}(k)$, $k \geq 0$, \ie it admits nontrivial sections then $L$ is nef.
\item \cite{matiso} Assume $B \cong \IP^n$ and $A$ ample. If $\kl$ is the universal line bundle of the Kuranishi family $\gX_L \rightarrow \Def(X,L)$ and $D$ a representative of $\Def(X,L)$, then the loci \begin{align*} \Dreg \ & = \ \left\lbrace t \in D \ | \ \varphi_{|\kl_t|} : \gX_t \rightarrow |\kl_t| \text{ is a Lagrangian fibration}  \right\rbrace \ , \\
\Dnef \ & = \ \left\lbrace t \in D \ | \ \kl_t \text{ is nef} \right\rbrace \end{align*}
coincide and are open and dense in $D$. Here $\varphi_{|\kl_t|}$ denotes the induced map by the linear system $|\kl_t|$. 
\end{enumerate}\end{lem}
\begin{proof}\begin{enumerate}
\item This is an immediate consequence from Fujiki's relation \cite[Prop. 3.9]{huy1}.	
\item The pullback $L$ is an effective divisor class, hence $c_1(L)$ belongs to the boundary of the positive cone $\kc^+_X$. By Theorem \ref{mats} we know that $\rho(B) = 1$, therefore $A$ is ample (hence nef, cf. \cite[1.4.1]{larza}), since $A$ admits a nontrivial section. Therefore, if $C \subset X$ is a curve then $L \cdot C = \deg(f^{\star}A|_C) \geq 0$, in particular for any rational curve $C$. By \cite[Prop. 3.2]{huycone} this implies that $c_1(L)$ is contained in the closure of the Kähler cone \ie it is nef.

Assume $X$ to be of $\kdrei$--type. If $L = f^ {\star} \ko_{\IP^n}(1)$ is not primitive, then write $L = k L'$ with $k > 1$ and $L'$ primitive. The line bundle $L'$ is isotropic and nef since $L$ is isotropic by (i) and nef by the first statement. By \cite[Thm. 1.2, Cor. 1.1]{matiso} in combination with \cite[Thm. 1.3, Rmk. 1.8]{eyal2} the induced map $\varphi_{|L'|} \colon X \rightarrow |L'|  = \IP^n$ by $|L'|$ is a Lagrangian fibration. Clearly we have $L' = \varphi_{|L'|}^\star \ko_{\IP^n}(1)$, hence $L = \varphi_{|L'|}^\star \ko_{\IP^n}(k)$. Since $\varphi_{|L'|}$ is surjective, we get
$$ n + 1 \ = \ h^0(X, L) \ \geq \ h^0(\IP^n, \ko_{\IP^n}(k)) \ = \ \begin{pmatrix} n + k \\ n \end{pmatrix} \ > \ n + 1 \, , $$
a contradiction. 

\item By \cite[Thm. 1.2]{matiso} the locus $\Dreg$ is open and dense in $D$. By (ii) we have $\Dreg \subset \Dnef$. The reverse inclusion follows by \cite[Claim 3.2]{matiso}.  
\end{enumerate}
\end{proof}

As a consequence of (iii) of the Lemma above $\gU^{\circ}_{\lambda^\bot}$ is locally isomorphic to $\Def(X,L)$. 

\begin{prop}\label{evm} Let $(X, \eta)$ be a marked irreducible holomorphic symplectic manifold of $\kdrei$--type, $f : X \rightarrow \IP^n$ be a Lagrangian fibration on $X$ and $L = f^{\star} \ko_{\IP^n}(1)$. Set $\lambda \coloneqq \eta(c_1(L))$. Then $\gU^{\circ}_{\lambda^\bot}$ and  $\Def(X,L)$ are locally isomorphic around $(X,\eta)$ and $0$ respectively. \end{prop}
\begin{proof} By construction, the moduli of marked pairs of $\kdrei$--type $\gM_{\Lambda}$ is locally around $(X,\eta)$ isomorphic to the open set $\Def(X)$. This identification restricts to an open inclusion $\Def(X,L) \subset \gM_{\lambda^\bot}$. Conclude by using Lemma \ref{liso} (iii) to see that locally around $0$ and $(X,\eta)$ the spaces $\Def(X,L)$ and $\gU^\circ_{\lambda^\bot}$ respectively coincide. \end{proof}

\begin{theo}\label{modulil} Let $\lambda$ be a primitive and isotropic element in the $\kdrei$--lattice $\Lambda$. The space $\gU^{\circ}_{\lambda^\bot}$ has the following properties. \begin{enumerate} 
\item It parametrizes isomorphism classes of marked pairs $(X,\eta)$ of $\gM^{\circ}_{\Lambda}$ with $X$ of $\kdrei$--type admitting a Lagrangian fibration $f : X \rightarrow \IP^n$ such that $$\eta\left(c_1\left(f^{\star} \ko_{\IP^n}(1)\right) \right) \ = \ \lambda \, .$$
\item It is open in $\gM^{\circ}_{\lambda^\bot}$ and connected. 
\item It is smooth and of dimension $20$. \end{enumerate}\end{theo}
\begin{proof} \begin{enumerate} 
\item Let $(X, \eta) \in \gU^{\circ}_{\lambda^\bot}$. As $H^1(X,\ko_X) = 0$ the exponential sequence on $X$ is 
$$ \cdots 0 \longrightarrow \Pic(X) \stackrel{c_1}{\longrightarrow} H^2(X,\IZ) \longrightarrow H^2(X,\ko_X) \longrightarrow \cdots \, .$$
Since $\eta^{-1}(\lambda)$ is of type $(1,1)$ it is in the kernel of $H^2(X,\IZ) \rightarrow H^2(X,\ko_X)$ so we can find an unique line bundle $L$ on $X$ such that $c_1(L) = \eta^{-1}(\lambda)$. Then by \cite[Thm. 1.2, Cor. 1.1]{matiso}  in combination with \cite[Thm. 1.3, Rmk. 1.8]{eyal2} $L$ induces a Lagrangian fibration $f : X \rightarrow |L^{\star}| = \IP^n$ since $\eta^{-1}(\lambda)$ is nef by assumption. %Note that you can drop the non--speciality assumption in \cite[Thm. 1.3]{eyal2} by \cite[Rmk. 1.8]{eyal2}.

Conversely let $(X,\eta)$ be a marked pair with $X$ of $\kdrei$--type admitting a Lagrangian fibration $f : X \rightarrow \IP^n$. Then $\lambda \coloneqq \eta(c_1(f^{\star}\ko_{\IP^n}(1)))$ is primitive and isotropic by Lemma \ref{liso} (i) and (ii). We then have $\kp(X,\eta) \in \Omega^{+}_{\lambda^\bot}$ \ie $(X,\eta)$ is in $\gM^{\circ}_{\lambda^\bot}$. In particular $\eta^{-1}(\lambda) = c_1(f^{\star}\ko_{\IP^n}(1))$ is nef by Lemma \ref{liso} (ii). This implies that $(X,\eta)$ is in $\gU^{\circ}_{\lambda^\bot}$. 
\item By Proposition \ref{evm} and Lemma \ref{liso} (iii) $\gU^{\circ}_{\lambda^\bot}$ is open. Let $(X,\eta)$ be in $\gM^{\circ}_{\lambda^\bot} \setminus \gU^{\circ}_{\lambda^\bot}$, then by definition $\eta^{-1}(\lambda)$ is not nef. By \cite[Prop. 3.14.]{ulrike1} there exists a monodromy operator $g \in \Mon^2(X)$ such that $g$ preserves the Hodge structure of $H^2(X,\IZ)$ and $g(\eta^{-1}(\lambda)) \in \overline{\kb \kk}_X$ where the latter denotes the closure of the birational Kähler cone $\kb \kk_X$. By \cite[Cor. 1.5]{eyalyoshi} there exists a bimeromorphic map $\phi : X \rightarrow X'$ where $X'$ is irreducible holomorphic symplectic such that $g(\eta^{-1}(\lambda)) = \phi^{\star} \alpha$ where $\alpha \in \overline{\kk}_X$ is nef on $X'$. Then $\eta' := \eta \circ g^{-1} \circ \phi^{\star}$ is a marking on $X'$ and $\eta'^{-1}(\lambda) = \alpha$ is nef, hence the pair $(X', \eta')$ is contained in $\gU^{\circ}_{\lambda^\bot}$. Since $g$ preserves the Hodge structure we have in particular $P(X,\eta) = P(X', \eta')$ for the periods. By M. Verbitsky's Global Torelli \cite[Thm. 4.24]{Verb11} the pairs $(X, \eta)$ and $(X', \eta')$ are inseparable points of $\gM^\circ_{\Lambda}$. This shows that the Hausdorffization of $\gM^\circ_{\lambda^\bot}$ coincide with the Hausdorffization of $\gU^\circ_{\lambda^\bot}$ which therefore must be connected since $\gM^{\circ}_{\lambda^\bot}$ is connected. We conclude that $\gU^\circ_{\lambda^\bot}$ is connected as its Hausdorffization is.

%by Lemma \ref{liso} (iii) the complement $\gM^{\circ}_{\lambda^\bot} \setminus \gU^{\circ}_{\lambda^\bot}$ is the union of countably many analytic subsets \ie $\gU^{\circ}_{\lambda^\bot}$ is the complement of a real codimension two subset in the connected space $\gM^{\circ}_{\lambda^\bot}$. Hence $\gU^{\circ}_{\lambda^\bot}$ must be connected by \cite[Lem. 4.10]{Verb11}. 
\item This was shown in Proposition \ref{evm}.
\end{enumerate} \end{proof} 

In the following we deal with the question: when do two Lagrangian fibrations on $\kdrei$--type manifolds lie in the same connected component  $\gU^{\circ}_{\lambda^\bot}$?

\begin{defin}\cite[5.2]{eyalprime}\label{pairdef} Let $X_i$, $i = 1, 2$, denote two irreducible holomorphic symplectic manifolds and $L_i$ line bundles on $X_i$. The pairs $(X_1, L_1)$ and $(X_2, L_2)$ are called \emph{deformation equivalent} if there exist a family $\pi : \kX \rightarrow S$ of irreducible holomorphic symplectic manifolds over a connected complex space $S$ with finitely many irreducible components, a section $e$ of $R^2\pi_{\star} \IZ$ which is of Hodge type $(1,1)$ everywhere, points $t_i$ in $S$ such that $\kX_{t_i} = X_i$ and $e_{t_i} = c_1(L_i)$. \end{defin}

\begin{prop}\label{defo} Let $f_i : X_i \rightarrow \IP^n$, $i = 1,2$, denote two Lagrangian fibrations with $X_i$ of $\kdrei$--type and set $L_i \coloneqq f_i^{\star}\ko_{\IP^n}(1)$. Then the following statements are equivalent.
\begin{enumerate}
 \item The Lagrangian fibrations $f_i$ are deformation equivalent in sense of Definition \ref{deflf}. 
 \item The pairs $(X_i, L_i)$ are deformation equivalent.
 \item There exist markings $\eta_i : H^2(X_i, \IZ) \rightarrow \Lambda$ such that
$$ \eta_1(c_1(L_1)) \ = \ \eta_2(c_1(L_2)) $$
and $\eta_2^{-1} \circ \eta_1$ is a parallel transport operator. 
\item There exist markings $\eta_i : H^2(X_i, \IZ) \rightarrow \Lambda$ such that the marked pairs $(X_i, \eta_i)$ are contained in the same connected component $\gU^{\circ}_{\lambda^\bot}$ for a primitive isotropic class $\lambda$ in the $\kdrei$--type lattice.
\end{enumerate}\end{prop}
\begin{proof} (i) $\Rightarrow$ (ii) Consider a family of Lagrangrian fibrations $\phi : \kX \rightarrow P$ over a complex space $S$ with points $t_i$ such that $\phi_{t_i} = f_i$ where we can assume that $P$ is a projective bundle as the $f_i$ are fibered over $\IP^n$. Let $\pi : \kX \rightarrow S$ denote the family of irreducible holomorphic symplectic manifolds belonging to the family $\phi$. Let $\kl \coloneqq \phi^{\star} \ko_P(1)$ then $e_t \coloneqq c_1(\kl|_{\kX_t})$ is clearly of Hodge type $(1,1)$ everywhere and defines a section of $R^2 \pi_{\star} \IZ$ such that $e_{t_i} = L_i$ hence the pairs $(X_i, L_i)$ are deformation equivalent.

(ii) $\Rightarrow$ (iii) Pick markings $\eta_i$ that only differ by a parallel transport operator induced by a family $\pi : \kX \rightarrow S$ as in the Definition \ref{pairdef}.

%Let $\pi : \kX \rightarrow S$ be a family of irreducible holomorphic symplectic manifolds with $S$ connected, $t_i$ points such that $\kX_{t_i} = X_i$ and $e$ a section of $R^2\pi_{\star}\IZ$ with $e_{t_i} = c_1(L_i)$. As $R^2\pi_{\star} \IZ$ is a local system we can find a neighbourhood $U$ of $t_2$ and a marking $\Sigma : R^2\pi_{\star}\IZ|_{U} \rightarrow \Lambda_{U}$. As $S$ is connected we can choose a path $\gamma$ connecting $t_1$ with $t_2$. Then $\gamma$ is parallel along $e$ \ie $\gamma^{\star} e$ is a flat section of $\gamma^{\star} R^2\pi_{\star}\IZ$. Consider the parallel transport $P_{\gamma} : H^2(X_1,\IZ) \rightarrow H^2(X_2, \IZ)$ along $\gamma$. Note that $P_{\gamma}(e_{t_1}) = e_{t_2}$. Define $\eta_2 \coloneqq \Sigma_{t_2}$ and $\eta_1 \coloneqq \eta_2 \circ P_{\gamma}$. Hence we have $\eta_2^{-1} \circ \eta_1 = P_{\gamma}$ and
%$$\eta_1(c_1(L_1)) \ = \ \eta_1(e_{t_1}) \ = \ \eta_2(P_{\gamma}(e_{t_1})) \ = \ \eta_2(e_{t_2}) \ = \ \eta_2(c_1(L_2)) \, .$$  

(iii) $\Rightarrow$ (iv) As $\eta_2^{-1}\circ \eta_1$ is a parallel transport operator the manifolds $X_i$ belong to a family of irreducible holomorphic symplectic manifolds. Hence the marked pairs $(X_i, \eta_i)$ belong to a connected component $\gM^{\circ}_{\Lambda}$ of the moduli of marked pairs. The condition that $\lambda \coloneqq \eta_1(c_1(L_1)) = \eta_2(c_1(L_2))$ then implies that the pairs $(X_i, \eta_i)$ are contained in $\gM^\circ_{\lambda^\bot}$, but also in $\gU^{\circ}_{\lambda^\bot}$ since the $L_i$ are nef and $\gU^{\circ}_{\lambda^\bot}$ is connected by Theorem \ref{modulil} (ii). 

(iv) $\Rightarrow$ (i) Choose a path $\gamma$ in $\gU^{\circ}_{\lambda^\bot}$ connecting the marked pairs $(X_i, \eta_i)$. We can choose finitely many points $x_1, \ldots, x_N$ which lie on $\gamma$ with the following properties
\begin{itemize}
	\item $(X_1,\eta_1) = x_1$ and $(X_2, \eta_2) = x_N$
	\item By Theorem \ref{modulil} each $x_k$ corresponds to a Lagrangian fibration $f_k : X_k \rightarrow \IP^n$ which belongs by Theorem \ref{matsh} to a family of Lagrangian fibrations $\zeta_k : \gX_k \rightarrow P_k$ parametrised by $\Def_k \coloneqq \Def(X_k, f_k^{\star}\ko_{\IP^n}(1))$. Therefore each $x_k$ admits the neighbourhood $\Def_k$ with $\Def_k \cap \Def_{k+1} \neq \emptyset$, for $k = 1, \ldots, N-1$.
	\item Note that $\gamma$ is covered by the $\Def_k$, $k = 1, \ldots, N$.  
\end{itemize}

%$(X_1,\eta_1) = x_1$ and $(X_2, \eta_2) = x_N$ such that every $x_k$ admits a neighbourhood $\Def_k$ which is the deformation space $\Def(X_i, E_i)$ of a pair $(X_i, E_i)$ corresponding to the point $x_i$, where $E_i$ is the line bundle given in Theorem \ref{modulil} (i). Further this neighbourhood $\Def_k$ is the parameter space of a family of Lagrangian fibrations $\zeta_k : \gX_k \rightarrow P_k$ over $\Def_k$ containing the Lagrangian fibration associated to the point $x_k$ and such that $\Def_{k} \cap \Def_{k+1} \neq \emptyset$ for $k = 1, \ldots, N-1$ as subsets of $\gU^{\circ}_{\lambda}$. 
Choose points in $z_k$ in $\Def_k \cap \Def_{k+1}$ for $k = 1, \ldots, N-1$. \begin{itemize} \item Set $S \coloneqq \coprod_{k=1}^N \Def_k/\sim$ where $\sim$ glues $\Def_k$ and $\Def_{k+1}$ at the point $z_k$ for $k = 1, \ldots, N-1$. 
\item Further set $\kX \coloneqq \coprod_{k=1}^N \gX_k / \sim$ where $\sim$ glues $\gX_k$ and $\gX_{k+1}$ at $(\gX_k)_{z_k}$ and $(\gX_{k+1})_{z_k}$. Note that those fibers are isomorphic. 
\item Denote by $\pi_k : \gX_k \rightarrow \Def_k$ the family of irreducible holomorphic symplectic manifolds belonging to the family $\zeta_k$ as in Theorem \ref{matsh}. Then the map $\pi : \kX \rightarrow S$ defined by $\pi|_{\gX_k} \coloneqq \pi_k$ is well defined and is a family of irreducible holomorphic symplectic manifolds. 
\item Set $P \coloneqq \coprod_{k=1}^N P_k / \sim$ where $\sim$ glues $P_k$ and $P_{k+1}$ at the projective spaces $(P_k)_{z_k}$ and $(P_{k+1})_{z_{k+1}}$ for $k = 1,\ldots,N-1$. We get a morphism $P \rightarrow S$ which is induced by the morphisms $P_k \rightarrow \Def_k$, $k = 1,\ldots, N$. This map is a family of projective spaces hence a projective bundle. \end{itemize}
Putting everything together we can define a map $\phi : \kX \rightarrow P$ locally given by the $\zeta_k : \gX_k \rightarrow P_k$, $k = 1,\ldots, N$. This defines by construction a family of Lagrangian fibrations over $S$ containing $f_1 = \phi_{x_1}$ and $f_2 = \phi_{x_2}$.
\end{proof}

\section{Polarization types of Lagrangian fibrations}\label{poltypes}

Let $X$ be an irreducible holomorphic symplectic manifold of dimension $2n$ and $f : X \rightarrow B$ a Lagrangian fibration. For a general point $t \in B$ the associated fiber $F \coloneqq f^{-1}(t)$ is an abelian variety even when $X$ is not projective. That $F$ is actually projective follows from  \cite[Prop. 2.1]{camp05}. A related statement of the latter is Proposition \ref{camp}. 

In this section it is explained how to associate to $f$ a tuple $\pol(f) \in \IZ^n$ of positive integers which is called the \emph{polarization type} of the fibration. This type $\pol(f)$ is the type of a polarization on $F$ in the classical sense.

\begin{defin} Let $F$ denote a smooth fiber. We say that a Kähler class $\omega \in H^{1,1}(X,\IR)$ is \emph{special with respect to} $F$ if the restriction $\omega|_F$ is integral \ie contained in $H^2(F,\IZ)$ and primitive \ie indivisible. We call such an $\omega$ just \emph{special} if there is no confusion with the fiber $F$. \end{defin}

\begin{exam} Of course every ample line bundle $\kl \in \Pic(X)$ defines a Kähler class $c_1(\kl) \in H^{1,1}(X,\IZ)$ which is integral on all fibers and for each smooth fiber $F$ we can find a natural number $k$ such that $\omega \coloneqq \frac{1}{k} c_1(\kl)$ is special with respect to $F$.
\end{exam}

The following Proposition is related to an observation of C. Voisin, see \cite[Proof of Prop. 2.1]{camp05}. 

\begin{prop} \label{camp} For every smooth fiber $F$ there is a Kähler class $\omega \in H^{1,1}(X,\IR)$ which is special with respect to $F$. \label{restriction} \end{prop}
\begin{proof} 	We have a surjective projection $p : H^2(X,\IR) \rightarrow H^{1,1}(X,\IR)$ which is induced by the Hodge decomposition. As $H^2(X,\IQ)$ is dense in $H^2(X,\IR)$ also $p(H^2(X,\IQ))$ is dense in $H^{1,1}(X,\IR)$. Since the Kähler cone $\kk_X$ is open in $H^{1,1}(X,\IR)$ we have $p(H^2(X,\IQ)) \cap \kk_X \neq \emptyset$ so that we can find a class $\alpha \in H^2(X,\IQ)$ with $p(\alpha) \in \kk_X$. 
	Denote by $r : H^2(X,\IR) \rightarrow H^2(F,\IR)$ the restriction. As $F$ is Lagrangian and $H^{2,0}(X)$ is generated by the holomorphic symplectic form the restriction $r_{H^{2,0}} : H^{2,0}(X) \rightarrow H^{2,0}(F)$ on holomorphic $2$--forms is zero, hence the non--$(1,1)$ parts of $\alpha$ are in the kernel of $r$ so we have $r(\alpha) = r( p(\alpha))$.
Then take a positive number $m > 0$ such that $m r(\alpha) \in H^2(F,\IZ)$ is integral and primitive. Consequently $\omega \coloneqq m p(\alpha)$ is a special Kähler class on $X$ with respect to $F$ since $r(\omega) = m r(p(\alpha)) = m r(\alpha) \in H^2(F,\IZ)$. \end{proof}

A \emph{polarization} on an abelian variety $A$ is by definition the first Chern class $c_1(L) \in H^2(A,\IZ)$ of an ample line bundle $L$ on $A$. The restriction $\omega|_F$ of a Kähler class which is special with respect to the smooth fiber $F$ defines a primitive polarization on the abelian variety $F$.

\begin{lem}\label{ray} Let $\kk_X$ be the Kähler cone of $X$, $F$ a smooth fiber and $r : H^2(X,\CC) \rightarrow H^2(F,\CC)$ the restriction. 
\begin{enumerate}
\item Then $\rk r = 1$, and
\item $G \coloneqq r(\kk_X) \subset H^{1,1}(F,\IR)$ is a ray that contains integral points. 
\end{enumerate} \end{lem}

\begin{proof} \begin{enumerate}
\item By \cite[Prop 1.2, Lemma 1.5]{Voi92} the space
$$ D_F \ \coloneqq \ \left\{ t \in \Def(X) \ | \ \text{there exists a deformation } \kf_t \subset \gX_t \text{ of } F \right\} \, .$$
is a complex submanifold of $\Def(X)$ and for its codimension in $\Def(X)$ one has $\codim D_F = \rk r$.

Let $L$ be the pullback of a very ample line bundle on $B$ by $f$. Then $\Def(X,L) \subset D_F$: By Theorem \ref{matsh} we have a family $\zeta : \gX_L \rightarrow P := \IP(\pi_{\star}L)$ of Lagrangian fibrations over $\Def(X,L)$ where $\pi : \gX_L \rightarrow \Def(X,L)$ is the universal family of the pair $(X,L)$ such that $\zeta_0 = f$. Let $F = X_{t_0}$ be the fiber over the point $t_0 \in B = P_0$. Then we can choose a neighbourhood $U$ of $0$ and a local holomorphic section $s : U \rightarrow P$ of the $\IP^n$--fibration $\gX_L \rightarrow P$ such that $s(0) = t_0$. Then the fiber product $\kf \coloneqq U \times_P \gX_L$ gives a deformation $\pr_U : \kf \rightarrow U$ of $\kf_0 = F$ hence $U \subset \Def(X,L)$ \ie $\Def(X,L) \subset D_F$ as germs. For the codimensions in $\Def(X)$ we have in particular
$$1 = \ \codim \Def(X,L) \ \geq \ \codim D_F \ = \ \rk r \ \geq \ 1$$ as a Kähler class on $X$ restricts to a nontrivial element in $H^2(F,\CC)$. We conclude that $\rk r = 1$.

\item By (i) we have also $\rk( r : H^{1,1}(X,\IR) \rightarrow H^{1,1}(F,\IR)) = 1$. As $\kk_X$ is open in $H^{1,1}(X,\IR)$ it follows that $\dim G = 1$. Since restrictions of Kähler classes are still Kähler classes, $G$ is a ray. By Proposition \ref{restriction} $G$ contains integral points. \end{enumerate}\end{proof}

\begin{rem} \begin{enumerate}
\item If $\pi : \gX \rightarrow \Def(X)$ denotes the Kuranishi family then the local system $R^2 \pi_{\star} \CC_{\gX}$ is trivial as we assume $\Def(X)$ to be simply connected. By Ehresmann's theorem we can choose a differentiable trivialization 
$$  \xymatrix{ \xymatrixcolsep{3pc} \gX  \ar[dr] & & X \times \Def(X) \ar[ll]_{\rho} \ar[dl] \\ 
										&	\Def(X) & 		   }  $$
where we denote by $\rho_t \coloneqq \rho|_{X} : X \rightarrow \gX_t$ the associated fiber diffeomorphism. Further we can choose a relative holomorphic form $\sigma$ \ie a section of $\Omega^2_{\gX/\Def(X)}$ such that the restriction $\sigma_t \coloneqq \sigma|_{\gX_t}$ is a holomorphic symplectic form on $\gX_t$. Then the space $D_F$ in the proof of Lemma \ref{ray} can also be defined as
$$ D_F \ = \ \left\{ t \in \Def(X) \ | \ r[\rho_t^{\star} \sigma_t] = 0 \in H^2(F,\CC) \right\} \, , $$
see \cite[Thm 0.1]{Voi92}. %Hence the space $D_F$ parametrizes deformations $\kf_t \subset \gX_t$ of $F$ which stay Lagrangian with respect to the Lagrangian fibration $\zeta_t : \gX_t \rightarrow P_t$. 
 \item From the proof it also follows that $D_F = \Def(X,L)$ as germs as $D_F$ is irreducible and contained in $\Def(X,L)$, but both have codimension one in $\Def(X)$.
\end{enumerate} \end{rem}

Let $\Delta \subset B$ be the \emph{discriminant locus} of the Lagrangian fibration $f : X \rightarrow B$, which is by definition the set parametrizing the singular fibers. Note that in general $\Delta$ is a reducible hypersurface in $B$, see \cite[Prop 3.1]{HO09}. Then $B^{\circ} \coloneqq B - \Delta$ is a connected open subset and the restriction $g \coloneqq f|_{f^{-1}(B^{\circ})} : f^{-1}(B^{\circ}) \rightarrow B^{\circ}$ is a proper holomorphic submersion. Let $\kc^{\infty}_{B^\circ}$ denote the sheaf of smooth real functions on $B^\circ$. By Ehresmann's theorem $\kh \coloneqq R^2 g_{\star} \IR \otimes \kc^{\infty}_{B^{\circ}}$ is a differentiable real vector bundle on $B^{\circ}$ which comes with a canonical flat connection $\nabla$ called the Gauss--Manin connection, see \cite[9.2.1]{voih1}.

For each $t \in B^{\circ}$ consider the restriction $r_t : H^2(X,\IR) \rightarrow H^2(X_t,\IR)$ where $X_t \coloneqq g^{-1}(t)$. Set $\kg_t \coloneqq r_t(\kk_X) \cap H^2(X_t, \IZ)$. By Lemma \ref{ray} $\kg_t$ is a non--empty semigroup of rank one. We can define a map \begin{equation}\label{alpha}\alpha \ : \ B^{\circ} \longrightarrow \kh \ \text{ such that } \ \alpha(t) \in \kg_t \subset H^2(X_t, \IR)\end{equation} is the unique integral and primitive element in $\kg_t$ for all $t \in B^{\circ}$.

\begin{prop}\label{alphaconst} The $\kg_t$ form a local system $\kg$ of semigroups on $B^{\circ}$. The map $\alpha : B^{\circ} \rightarrow \kh$ is continuous, and thus can be considered as a section of $\kg$. \end{prop}
\begin{proof} Consider the family of sections
\begin{equation}\label{phi}\varphi : \kk_X \times B^{\circ} \rightarrow \kh \, , \ \ (\omega, t) \mapsto \omega|_{X_t} \, .\end{equation}
Then the image of $\varphi$ is the union of rays in each $H^2(X_t,\IR)$ considered in Lemma \ref{ray} containing integral points. Note that the family $\varphi$ is differentiable as for each $\omega \in \kk_X$ the corresponding section $\varphi(\omega, \cdot)$ is differentiable. More precisely it is flat, \ie $\nabla \varphi(\omega, \cdot) = 0$ for each $\omega \in \kk_X$ which follows from the Cartan--Lie formula, see \cite[9.2.2]{voih1}. 

Let $\kh^{\nabla}$ be the sheaf of flat sections of $\kh = R^2 g_{\star} \IR \otimes \kc^{\infty}_{B^{\circ}}$. As $\varphi$ is a flat family the image $\im \varphi$ is a local system of semigroups which is contained in $\kh^{\nabla}$. Then define $\kg \coloneqq \im \varphi \cap R^2g_{\star} \IZ$ which is in a canonical way a local system whose stalks are given precisely by $\kg_t$. 

Take an open cover $B^{\circ} = \cup_i U_i$ such that $\kg$ is trivial on each $U_i$ say $\kg(U_i) = G$ for all $i$ where $G \coloneqq \kg_t$ for a fixed $t$. For each $i$ the restriction $\alpha|_{U_i}$ is the unique primitive element in $G$. They glue to an unique global section of $\kg$ which is precisely $\alpha$. Hence $\alpha$ is continuous as a section of the local system $\kg$ and in particular as a map $B^{\circ} \rightarrow \kh$. \end{proof}

Clearly $\alpha(t) \in H^2(X_t,\IZ)$ defines a polarization on the abelian variety $X_t$ for every $t \in B^{\circ}$. To any polarization on an abelian variety one can associate a tuple of positive integers which is called the \emph{polarization type}, see \cite[p. 70]{BL}, in the following way. 

Since $X_t$ is an abelian variety we have an identification $H^2(X_t,\IZ) = \wedge^2 H_1(X_t,\IZ)^{\vee}$, see \cite[Cor. 1.3.2]{BL}. We can interpret $\alpha(t) : \Lambda_t \otimes \Lambda_t \rightarrow \IZ$ as an alternating integral form on the lattice $\Lambda_t \coloneqq H_1(X_t,\IZ)$. By the elementary divisor theorem we can find a basis of $\Lambda_t$ for which $\alpha(t)$ has the form
$$ \alpha(t) \ = \ \begin{pmatrix} \phantom{-}0 & D \\ -D & 0 \end{pmatrix}  $$ 
where $D = \text{diag}(\lambda_1, \ldots \lambda_n)$ is an integral diagonal matrix with $\lambda_i > 0$ and $\lambda_i | \lambda_{i+1}$. The tuple $$\pol(f,t) \ \coloneqq \ (\lambda_1, \ldots, \lambda_n)$$ is called the \emph{polarization type} of $\alpha(t)$ and a priori depends on $t \in B^{\circ}$. We also use the notation $\pol(L)$ for the type of a polarization $L$ on an abelian variety in the classical sense \cite[p. 70]{BL} \ie $\pol(f,t) = \pol(\alpha(t))$.

\begin{prop}\label{polcon} The polarization type $\pol(f, \cdot) : B^{\circ} \rightarrow \IZ^n$ is constant. \end{prop}
\begin{proof} By construction for $t \in B^{\circ}$ the associated tuple $\pol(f,t)$ is the diagonal of one of the blocks of the representation matrix of $\alpha(t) : \Lambda_t \times \Lambda_t \rightarrow \IZ$ with respect to a chosen basis $b_1(t), \ldots, b_{2n}(t)$ of the lattice $H_1(X_t,\IZ)$. This correspondence is continuous and since $\pol(f, \cdot)$ is integer valued in each component it is locally constant hence constant as $B^{\circ}$ is connected. \end{proof}

\begin{defin} For each Lagrangian fibration $f : X \rightarrow B$ the associated tuple $\pol(f)$ in $\IZ^n$ is called the \emph{polarization type} of $f$.\end{defin}

\begin{theo}\label{samedef} The polarization type stays constant in a family of Lagrangian fibrations. In particular two Lagrangian fibrations which are deformation equivalent as Lagrangian fibrations have the same polarization type. \end{theo}
\begin{proof} The proof is similar to the one of Proposition \ref{alphaconst}. Let $\phi : \kX \rightarrow P$ be a family of Lagrangian fibrations parametrized by a complex space $S$. Setting $\kb \coloneqq \bigcup_{s\in S} B^{\circ}_{s}$ where as before $B^{\circ}_s \coloneqq P_s - \Delta_s$ is the base of the Lagrangian fibration $\phi_s \coloneqq \phi|_{\kX_s} : \kX_s \rightarrow P_s$ without the discriminant locus. Note that $\kb$ is connected as it is $P$ without a real codimension two subset. Set $\psi \coloneqq \phi|_{\pi^{-1}(\kb)} : \phi^{-1}(\kb) \rightarrow \kb$ which is a holomorphic submersion.  

With same argument as in Proposition \ref{alphaconst} we get a section $A : \kb \rightarrow R^2 \psi_{\star} \IZ$ such that for $t \in B_s$ the value $A(t)$ coincides with $\alpha_s(t)$, where $\alpha_s$ is the continuous map $\alpha_s : B_s^{\circ} \rightarrow \kh_s$ as in \eqref{alpha} which is associated to the Lagrangian fibration $\phi_s$. Let $\pol(A(t))$ denote the polarization type of the polarization $A(t)$ on the abelian variety $(\kX_s)_t$ for $t \in B^{\circ}_s$. As $\kb$ is connected the continuous map $\pol(A(\cdot)) : \kb \rightarrow \IZ^n$ must be constant. Since $\pol(\phi_s) = \pol(\alpha_s(t)) = \pol(A(t))$ for $t \in B^{\circ}_s$ we see that $\pol(\phi_s)$ is constant on $S$. \end{proof} 

\begin{prop}\label{specialpol} Let $f : X \rightarrow B$ be a Lagrangian fibration and $\omega$ a special Kähler class with respect to a smooth fiber $F$. Then $\pol(f)$ is given by the polarization type of $\omega|_F$ \ie $\pol(f) = \pol(\omega|_F)$. \end{prop}
\begin{proof} As $\omega|_F$ is the restriction of a Kähler class it is contained in the ray $G$ of Lemma \ref{ray}. Since $\omega|_F$ is primitive it is in the image of $\alpha : B^{\circ} \rightarrow \kh$ of \eqref{alpha} \ie $\alpha(t) = \omega|_F$ for $F = X_t$. \end{proof}
						
\begin{exam} Let $f : S \rightarrow \IP^1$ be an elliptic K$3$ surface and $\omega$ a special Kähler class on $S$ with respect to a smooth fiber $F$. As $F$ is an elliptic curve we have $H^2(F,\IZ) \cong \IZ$. Since $\omega|_F$ is primitive it is the generator of $H^2(F,\IZ)$ and so $\omega|_F = c_1(L)$ for an ample line bundle of degree $\deg(L) = 1$. By the Proposition above and \cite[Lem. 3.6.4, Lem. 3.6.5]{BL} we have $$ \pol(f) \ = \ \pol(\omega|_F) \ = \ \int_F c_1(L) \ = \ \deg(L) \ =\ 1$$ as one can identify the degree with integration of the first Chern class.
\end{exam}

\begin{theo}\label{same} The associated Lagrangian fibrations of two marked pairs which define points in the same connected component $\gU^{\circ}_{\lambda^\bot}$ of the moduli of Lagrangian fibrations of $\kdrei$--type for a primitive isotropic class $\lambda$ in the $\kdrei$--type lattice have the same polarization type.
\end{theo}
\begin{proof} By Proposition \ref{defo} the associated Lagrangian fibrations are deformation equivalent and the claim follows by Proposition \ref{samedef}.\end{proof}

A very optimistic conjecture is the following.

\begin{conj}\label{conjpol} Let $f_i : X_i \rightarrow B_i$, $i = 1, 2$, be two Lagrangian fibrations such that $X_1$ and $X_2$ are deformation equivalent. Then $\pol(f_1) = \pol(f_2)$. \end{conj}

As every known irreducible holomorphic symplectic manifold can be deformed to one which admits a Lagrangian fibration this conjecture would basically give a new deformation invariant. In section \ref{kdrei} we verify this conjecture for manifolds of $\kdrei$--type. 

\section{Beauville--Mukai systems}\label{beauvillemukai}

This section has the purpose to recall the notion of Beauville--Mukai systems which are important examples of $\kdrei$--type fibrations and to determine their polarization type. \\

Let $S$ be a projective K$3$ surface and let $H^{\bullet}(S)$ denote the \emph{Mukai lattice} \ie 
$$ H^{\bullet}(S) \ \coloneqq \ H^0(S,\IZ) \oplus H^2(S, \IZ) \oplus H^4(S,\IZ)$$
together with the bilinear form defined by $(v,w) \coloneqq (v_2,w_2) - \int_S (v_0 \wedge w_4 + v_4 \wedge w_o)$ where $(v_2, w_2) = \int_S v_2 \wedge w_2$ denotes the intersection form on $H^2(S,\IZ)$ and $v = v_0 + v_2 + v_4$ with $v_{i} \in H^i(S,\IZ)$ the decomposition in $H^{\bullet}(S)$ and similarly for $w$. This lattice is even, unimodular, of rank $24$ and isometric to
\begin{equation}\label{mukailattice} \tilde{\Lambda} \ \coloneqq \ \Lambda_{\text{K}3} \oplus U \ = \ E_8(-1)^{\oplus 2} \oplus U^{\oplus 4} \end{equation}
where $\Lambda_{\text{K}3} \cong H^2(S,\IZ)$ is the the K$3$ lattice, $E_8(-1)$ the negative definite root lattice of type $E_8$ and $U$ the unimodular rank two hyperbolic lattice. We identify $H^4(S,\IZ) = \IZ$ where we use the Poincare dual to a point as a generator and similarly $H^0(S,\IZ) = \IZ$ by taking the Poincare dual of $S$.

 A \emph{Mukai vector} is a tuple $v = (r,l,s) \in H^0(S,\IZ) \oplus H^{1,1}(S,\IZ) \oplus H^4(S,\IZ)$ such that $r \geq 0$ and $l$ is effective if $r = 0$. For a coherent sheaf $\kf \in \coh(S)$ set $v(\kf) \coloneqq \ch(\kf)\sqrt{\td(S)} \,$ which is a Mukai vector as easily verified\footnote{Note that $v(\kf) = (\rk(\kf), c_1(\kf), \chi(\kf) - \rk(\kf) = (\rk(\kf), c_1(\kf), \frac{1}{2} c_1(\kf)^2 - c_2(\kf) + \rk(\kf))$ as $\sqrt{\td(S)} = 1 + \omega$ where $\omega$ denote the Poincare dual of a point since we consider sheaves on a K$3$ surface.}.
 
 By \cite[4.C]{huylehn} for each Mukai vector $v$ there is a distinguished countable system of hyperplanes called \emph{$v$--walls} in the the real ample cone $A(S) \otimes \IR$ which is locally finite. An ample divisor $H$ on $S$ is called \emph{$v$--generic} if it lies outside the union of such $v$--walls. 
 
 Choose a Mukai vector $v$ together with a $v$--generic ample class $H$ and consider the moduli space $M_H^{s}(v)$ of $H$--stable pure coherent sheaves $\kf$ on $S$ with $v(\kf) = v$. Then general results of many authors, see for instance \cite{mukais} and \cite{yoshi}, imply that $M_H^{s}(v)$ is a holomorphic symplectic manifold of dimension $2n \coloneqq (v,v) + 2$. It is not necessarily compact. To compactify one needs to add semistable sheaves but under the assumption that $v$ is primitive any $H$--semistable sheaf $\kf$ with $v(\kf) = v$ is automatically $H$--stable and we denote the moduli space in this case just by $M_H(v)$ which is a $\kdrei$--type manifold, see \cite[Prop. 4.12]{yoshi}. 
\begin{exam} Let $S$ be a projective K$3$ surface, and $\kf$ a sheaf on $S$ such that $C \coloneqq \supp(\kf)$ is a smooth irreducible curve of genus $g$ with $\rk(\kf|_C) = 1$. Then $\rk(\kf) = 0$ and $c_1(\kf) = c_1(\ko_S(C))$. We have $H^i(C, \kf|_C) = H^i(S, \kf)$ (see this for instance with {\v C}ech cohomology), hence $\chi(\kf|_C) = \chi(\kf)$. The Riemann--Roch for curves gives $$v(\kf) \ = \ (0, c_1(\kf), \chi(\kf) - \rk(\kf)) \ = \ (0, C, 1 - g + d)$$ where $d$ denotes the degree of the restriction of $\kf$ to $C$. \end{exam}
Let $v$ be a primitive Mukai vector on $S$ of the form $v = (0, c_1(D), s)$ where $D$ is a big and nef divisor on $S$. Note that we have $h^0(S,D) = \frac{1}{2} (D,D) + 2 = n + 1$. Choose a $v$--generic ample class $H$ on $S$, hence $M_H(v)$ is an irreducible holomorphic symplectic manifold. It comes with a natural Lagrangian fibration in the following way. 

The space $M_H(v)$ parametrizes sheaves $\kf$ on $S$ with $v(\kf) = v$ \ie $\kf$ is of rank zero with Chern class $c_1(\kf) = c_1(D)$. In particular $\kf$ is supported on a divisor which is an element of $|D| \cong \IP^n$. Then we can define
$$\pi : M_H(v) \longrightarrow |D|^{\star}\, , \ \ \ \kf \longmapsto \supp \kf$$
to obtain a holomorphic map $M_H(v) \rightarrow |D|^{\star}$. Since $M_H(v)$ is irreducible holomorphic symplectic $\pi$ is a Lagrangian fibration by Matsushita's Theorem \ref{mats}. If $C$ is a smooth curve and an element of $|D|$ then the fiber $\pi^{-1}(C)$ is the Jacobian of the curve $C$ by construction. 

\begin{defin}Let $S$ be a projective K$3$ surface. Fix a big and nef divisor $D$ on $S$, a primitive Mukai vector of the form $v = (0, c_1(D), s)$ and a $v$--generic polarization $H$. The Lagrangian fibration $\pi : M_H(v) \rightarrow |D|^{\star}$, $\kf \mapsto \supp \kf$ as above is called a \emph{Beauville--Mukai system}.\end{defin} 

\begin{rem} 
More classically Beauville--Mukai systems arise from linear systems induced by a smooth curve in $S$ in the following way. Let $C \subset S$ denote an irreducible smooth curve of genus $n$. Under the assumption that $\Pic(S)$ is generated by $\ko_S(C)$ all curves in the linear system $|C|$ are reduced and irreducible. By Riemann--Roch it follows that $|C| = \IP^n$. Let $\kc \rightarrow \IP^n$ denote the associated family of curves. For each $d$, the relative compactified Jacobian $\pi : X \coloneqq \overline{\Pic}^d(\kc/\IP^n) \rightarrow \IP^n$ exists, see \cite[II, 1-4]{souza} or \cite[Thm. 6.6]{compac}. 

Setting $v \coloneqq (0, c_1(C), d + 1 -n)$ there is an identification of $\overline{\Pic}^d(\kc/\IP^n)$ with $M_H(v)$, see \cite[Ex. 0.5]{mukais}, given by the following map. For a pair $(\kc_t, \kf)$ representing an element in $X$, consider the inclusion $\iota : \kc_t \hookrightarrow S$. Then associate to it the element $\iota_{\star} \kf \in M_H(v)$. 

In particular one can see $M_H(v)$ as a generalization of the classical definition of a Beauville--Mukai system since the construction with the compactified Picard scheme only works if the linear system contains only reduced and irreducible curves. \end{rem}

\begin{lem}\label{picard} Let $A$ be an abelian variety. \begin{enumerate}
\item If $\End(A) = \IZ$ then its Picard number is $\rho(A) = 1$.
\item If $A = \Jac(C)$ is a Jacobian of a smooth curve $C$ and $\rho(A) = 1$ then the primitive polarization $\Theta$ on $A$ is principal \ie $\pol(\Theta) = (1, \ldots, 1)$.
\end{enumerate} \end{lem} 
\begin{proof} \begin{enumerate} \item By  \cite[Prop. 5.2.1]{BL} there is an isomorphism $\NS(A) \otimes \IQ \cong V$ where $V$ is a $\IQ$--subspace of $\End(A)\otimes \IQ$. The latter has dimension $1$ by assumption hence $\rho(A) = \dim_{\IQ} \NS(A) \otimes \IQ = 1$.
\item It is well known that on every Jacobian of a curve there exists a primitive principal polarization, see \cite[Prop 11.1.2]{BL}. Since $\rho(A) = 1$ it must be unique.
%\item This follows directly from \cite[Thm 1.1, Cor. 1.2]{cili} since K$3$ surfaces satisfy $H^1(S,\ko_S) = 0$. 
\end{enumerate} \end{proof}

\begin{theo}\label{picardk}\cite[Thm. 1.1, Cor. 1.2]{cili} Let $S$ be a projective K$3$ surface and $V$ a linear system on it. If $C$ is a general element of $V$ and $\Jac(C)$ the Jacobian of $C$ then $\End(\Jac(C)) = \IZ$.\end{theo}
\begin{proof} This follows directly from \cite[Thm. 1.1, Cor. 1.2]{cili} since K$3$ surfaces satisfy $H^1(S,\ko_S) = 0$. One has to note that the condition in \cite[Thm. 1.1]{cili} that $V$ defines a birational map from $S$ to its image can be dropped because the authors only use this to conclude that the pullback morphism 
	$$ \Pic^0(S) \longrightarrow \Pic^0(C) = \Jac(C) $$ 
has finite kernel, see \cite[p. 35, 2.II.]{cili}. Since $S$ is K$3$ we have $\Pic^0(S) = 0$ and so this condition is satisfied. \end{proof}

\begin{cor}\label{prin} The Picard number of the generic smooth fiber of a Beauville--Mukai system $\pi : X \rightarrow |D|^{\star}$ equals one. In particular $\pol(\pi) = (1, \ldots, 1)$. \end{cor}
\begin{proof} The first statement follows immediately from Theorem \ref{picardk} and Lemma \ref{picard} (i).  %since the generic smooth fiber is a Jacobian of a smooth curve which is an element of the linear system $|D|$. 
	A special Kähler class $\omega$ on $M_H(v)$ with respect to a fiber $F$ which is a Jacobian of a curve restricts to the unique primitive polarization on $F$ since $\rho(F) = 1$. Then by Lemma \ref{picard} (ii) above this polarization is principal, hence $\pol(\pi) = \pol(\omega|_F) = (1,\ldots,1)$ by Proposition \ref{specialpol}. \end{proof}

\section{Polarization type of $\kdrei$--type fibrations}\label{kdrei}

In this section we verify Conjecture \ref{conjpol} for $\kdrei$--type manifolds which follows by methods developed by E. Markman \cite{eyal2}.

\begin{theo}\label{main} Let $f : X \rightarrow \IP^n$ be a Lagrangian fibration with $X$ of $\kdrei$--type. Then the polarization type $\pol(f)$ is principal \ie it is given by $\pol(f) = (1, \ldots, 1)$. \end{theo}

For the following see also section $2.$ of \cite{eyal2}.

\begin{defin} Let $\lambda$ be an element in a lattice $\Lambda$. Then define the \emph{divisibility} of $\lambda$ as
$$ \Div(\lambda) \ \coloneqq \ \max\left\{ k \in \IN \ | \ (\lambda, \cdot)/k \text{ is an integral class in the dual } \Lambda^{\vee} \right\} \, . $$ \end{defin}

Let $\Lambda$ denote the $\kdrei$--type lattice and let $X$ be a $\kdrei$--type manifold. Fix a positive integer $d$ such that $d^2$ divides $n-1$. Next we associate to an isotropic and primitive class $\lambda \in \Lambda$ a lattice in the following way.

Let $\tilde{\Lambda}$ denote the Mukai lattice, see equation \eqref{mukailattice}. Then by \cite[Thm. 1.10]{eyal3} $X$ comes with a natural monodromy invariant $O(\tilde{\Lambda})$--orbit $\iota_X$ of primitive isometric embeddings $\iota : H^2(X,\IZ) \hookrightarrow \tilde{\Lambda}$. Choose a primitive isometric embedding $\iota : H^2(X,\IZ) \hookrightarrow \tilde{\Lambda}$ in $\iota_X$. Since $\iota(H^2(X,\IZ))$ is of rank $23$ and the Mukai lattice is of rank $24$ the orthogonal complement $\iota(H^2(X,\IZ))^{\bot}$ is of rank $1$. Choose a generator $v$ of $\iota(H^2(X,\IZ))^{\bot}$. Note that $(v,v) = 2n-2$. Then define the lattice $H(\lambda)$ to be the saturation of $\left\langle v, \iota(\lambda) \right\rangle \subset \tilde{\Lambda}$ \ie $H(\lambda)$ is the maximal sublattice of $\tilde{\Lambda}$ such that $H(\lambda)$ is of rank two and contains $\left\langle v, \iota(\lambda) \right\rangle$. 
Two pairs $(H_1, v_1)$ and $(H_2, v_2)$ are called \emph{isometric} if there is an isometry $g : H_1 \rightarrow H_2$ such that $g(v_1) = v_2$. The isometry class of $(H(\lambda),v)$ only depends on $\lambda$ since the $O(\tilde{\Lambda})$--orbit $\iota_X$ is monodromy invariant. 
Denote by $H_{n,d}$ the lattice $\IZ^2$ togeter with the pairing
$$ \frac{2n-2}{d^2} \begin{pmatrix} 1 & 0 \\ 0 & 0 \end{pmatrix} $$
and denote by $I_{n,d}$ the set of isometry classes of pairs $(H,w)$ such that $H$ is isometric to $H_{n,d}$ and $w$ is a primitive class in $H$ with $(w,w) = 2n-2$.
Given a positive integer $d$ let $I_d(X) \subset H^2(X,\IZ)$ be the subset of primitive isotropic classes $\lambda$ with respect to the Beauville--Bogomolov form such that $\Div(\lambda) = d$. Note that this set is clearly monodromy invariant \ie $\Mon^2(X) \cdot I_d(X) \subset I_d(X)$.  Note that if $X'$ is deformation equivalent to $X$ and $P : H^2(X,\IZ) \rightarrow H^2(X',\IZ)$ is a parallel transport operator then $I_d(X') = P(I_d(X))$.

\begin{lem}\cite[Lem. 2.5]{eyal2}\label{invariant} Let $\lambda$ denote a primitive isotropic class in $H^2(X,\IZ)$, let $\iota$ and $v$ as above and set $d \coloneqq \Div(\lambda)$.  \begin{enumerate}
\item The square $d^2$ divides $n-1$ and the lattice $H(\lambda)$ is isometric to $H_{n,d}$.
\item The map defined by $$ \vartheta \ : \ I_d(X) \longrightarrow I_{n,d} \, , \ \ \lambda' \longmapsto [(H(\lambda'), v)] $$ induces a bijection $I_d(X)/\Mon^2(X) \rightarrow I_{n,d}$.
\item For the pair $(H(\lambda), v)$ there exists an integer $b$ such that $(\iota(\lambda) - bv)/\Div(\lambda)$ is an integral class of $H(\lambda)$ and the isometry class $\vartheta(\lambda)$ of $(H(\lambda), v)$ is represented by $(H_{n,d}, (d,b))$ for any such integer $b$.  
\end{enumerate}
\end{lem}

\begin{rem}\label{bteilt} The integer $b$ in Lemma \ref{invariant} (iii) satisfies $\gcd(d,b) = 1$. Indeed, if they would have a common divisor say $r > 1$, then we could write the equation in Lemma \ref{invariant} (iii) as
$$ \iota(\lambda) \ = \ dm + bv \ = \ r(d'm + b'v) $$
for integers $m, d'$ and $b'$ which is a contradiction to the fact that $\iota(\lambda)$ is primitive. \end{rem}

The map $\vartheta : I_d(X) \rightarrow I_{n,d}$ is also called a \emph{monodromy invariant for $X$}, see \cite[5.3]{eyalprime} for the general notion.

\begin{lem}\cite[Ex. 3.1]{eyal2}\label{bm} Let $d$ be a positive integer such that $d^2$ divides $n-1$ and let $b$ an integer satisfying $\text{gcd}(d,b) = 1$. Then there exists a Beauville--Mukai system $\pi : M_H(v) \rightarrow \IP^n$ and a primitive isotropic class $\alpha \in H^2(M_H(v), \IZ)$ such that the following holds.
\begin{enumerate}
\item $\Div(\alpha) = d$,
\item the monodromy invariant $\vartheta(\alpha)$ is represented by $(H_{n,d}, (d,b))$,
\item $c_1(\pi^{\star}\ko_{\IP^n}(1)) = \alpha$. 
\end{enumerate}\end{lem}
\begin{proof} This is a more detailed and complete version of \cite[Ex. 3.1]{eyal2}. Let $S$ be a projective K$3$ surface together with a nef and primitive line bundle $L$ on $S$ of Bogomolov degree $(2n-2)/d^2$ e.g. take a $(2n-2)/d^2$--polarized K$3$ surface.  Set $\beta \coloneqq c_1(L)$ and let $s$ be an integer such that $sb \equiv 1 \mod d$. Then $v \coloneqq (0, d\beta, s)$ is a Mukai vector. In particular $v$ is primitive since $\beta$ is primitive and $\gcd(d,s) = 1$. Choose an $v$--generic ample class $H$. We have $(v,v) = d^2 (\beta,\beta) = 2n-2$ hence $M_H(v)$ is irreducible holomorphic symplectic of dimension $2n$ and we obtain a Beauville--Mukai system $\pi : M_H(v) \rightarrow |L^d|^{\star}$ as described in section \ref{beauvillemukai}. We have Mukai's Hodge isometry
$$\Theta \ : \ v^{\bot} \longrightarrow H^2(M_H(v), \IZ) $$
which can be defined as follows. Choose a quasi--universal family of sheaves $\ke$ on $S$ of similitude $\rho \in \IN$, cf. \cite[Thm. A.5]{mukaibundle}. That is $\ke \in \coh(S \times M_H(v))$ such that $\ke$ is flat over $M_H(v)$ and for every class $\kf \in M_H(v)$ one has $\ke|_{S \times \left\{\kf\right\}} \cong \kf^{\oplus \rho}$. Then set 
$$\Theta(x) \ \coloneqq \ \frac{1}{\rho} \left[(\pr_{M_H(v)})_{!}\left(( \ch(\ke) (\pr_{S})^{\star}(\sqrt{\td(S)} x^{\vee})\right) \right]_2 $$
where $x^{\vee} = -x_0 + x_2 + x_4$ for $x = x_0 + x_2 + x_4$ and $[ \cdot ]_2$ denotes the part in $H^2(S \times M_H(v),\IZ)$. For the details see \cite[1.2]{yoshi}. 

Set $\alpha \coloneqq \Theta(0,0,1)$ which is clearly isotropic and define $\iota : H^2(M_H(v),\IZ) \rightarrow H^{\bullet}(S,\IZ)$ to be $\Theta^{-1}$ composed with the inclusion $v^{\bot} \hookrightarrow H^{\bullet}(S,\IZ)$.
\begin{enumerate}
\item An element $(r,c,t)$ belongs to $v^{\bot}$ iff $$0 \  = \  ((0,d\beta,s), (r,c,t)) \ = \ d(\beta,c) - rs  \ \Longleftrightarrow \ rs \ = \  d(\beta,c) \, .$$
Hence $d$ divides $r$ since $\gcd(d,s) = 1$. Further we have $((0,0,1),(r,c,t)) = r$ for all $(r,c,t) \in v^{\bot}$ hence $\Div((0,0,1)) \geq d$. As the K$3$--lattice is unimodular we have $\Div_{H^2(S,\IZ)}(\beta) = 1$ in $H^2(S,\IZ)$. This implies that $\Div(\beta) = 1$ in $v^{\bot}$ hence we can find an element $c \in H^2(S,\IZ)$ such that $s = (c, \beta)$. Then $(d, c, 0)$ is contained in $v^{\bot}$ and $((0,0,1),(d,c,0)) = d$ hence 
$$\Div(\alpha) \ = \ \Div(0,0,1) \ = \ d \, .$$
\item We have $\iota(\alpha) - bv = (0,0,1) - (0,bd\beta, bs) = (0, bd\beta, 1 -bs)$ which is divisible by $d$ since $sb \equiv 1 \mod d$ by assumption. By Lemma \ref{invariant} (iii) the monodromy invariant $\vartheta(\alpha)$ is represented by $(H_{n,d}, (d,b))$.

\item Let $\omega = [p] \in H^4(S,\IZ)$ denote Poincare dual of a point $p \in S$. By our notation we have $\omega = (0,0,1) = \omega^{\vee} \in H^{\bullet}(S)$. Since $S$ is a K$3$ surface one has $\sqrt{\td(S)} = 1 + \omega$ hence $\sqrt{\td(S)} \omega = \omega$. Note that $\ke$ is a sheaf of rank zero hence $\ch(\ke) = \rho c_1(\ke) + \xi = \rho [D] + \xi$ for some divisor $D$ in $S \times M_H(v)$ and for some terms $\xi$ of higher degree. Further $(\pr_{S})^{\star} \omega = [p \times M_H(v)] \in H^4(S \times M_H(v), \IZ)$ and $[(\pr_{M_H(v)})_{!}(\xi \cdot [p \times M_H(v)])]_2 = 0$ due to degree reasons. Then we have 
\begin{align*} \Theta(0,0,1) \ & = \ (\pr_{M_H(v)})_{!}\left(D \cdot [p \times M_H(v)] \right) \\ \ & = \ [\kf \in M_H(v) \ | \ p \in \supp(\kf)] \\
\ & = \ \pi^{\star}[C \in |L^d| \ | \ p \in C] \\
\ & = \ \pi^{\star}c_1(\ko_{|L^d|}(1)) \ = \ c_1( \pi^{\star} \ko_{|L^d|}(1)) \end{align*} 
since $V \coloneqq \left\{ C \in |L^d| \ | \ p \in C \right\}$ is a hyperplane in a projective space, hence $[V] = c_1(\ko_{|L^d|}(1))$.  \end{enumerate} \end{proof}

For a fixed $\kdrei$--type manifold $X$ we have defined the monodromy invariant $\vartheta : I_d(X) \rightarrow I_{n,d}$. If $X'$ is another $\kdrei$--type manifold then we denote the monodromy invariant also by $\vartheta : I_d(X') \rightarrow I_{n,d}$.   

\begin{lem}\cite[Lem. 5.17]{eyalprime}\label{markcrit} Let $X_i$, $i = 1,2$, be two $\kdrei$--type manifolds and $e_i \in I_d(X_i)$. Assume that $\vartheta(e_1) = \vartheta(e_2)$, $e_i = c_1(L_i)$ for holomorphic line bundles $L_i$ and that there are Kähler classes $\kappa_i$ such that $(\kappa_i, e_i) > 0$. Then the pairs $(X_i, L_i)$ are deformation equivalent in the sense of Definition \ref{pairdef}. \end{lem}

Note that the Lemma above is stated in \cite{eyalprime} in a much more general setting, for monodromy invariants for irreducible holomorphic symplectic manifolds of arbitrary deformation type.

\begin{lem}\label{isopos} Let $\lambda$ be a nontrivial isotropic class in the closure $\bar{\kc}^{+}_X$ of the positive cone in $H^{1,1}(X,\IR)$ with $X$ an arbitrary irreducible holomorphic symplectic manifold. Then the Beauville--Bogomolov quadratic form satisfies $(x, \lambda) > 0$ for every class $x$ in the positive cone $\kc^{+}_X$.\end{lem}
\begin{proof} As $\kc^{+}_X$ is self--dual the cone coincides with its dual \ie $\kc^{+}_X = (\kc^{+}_X)^{\vee}$. This means $(x,y) > 0$ for all $y \in \kc^{+}_X$. Taking the closure of the positive cone this condition changes to $(x,y) \geq 0$, in particular $(x,\lambda) \geq 0$. 	
		%The condition $(\lambda, \lambda) = 0$ means nothing as that $\lambda$ is contained in the closure $\bar{\kc}^{+}_X$ of the positive cone hence $(x,\lambda) \geq 0$. 
	As $(x,x) > 0$ and the signature of the Beauville--Bogomolov form on $H^{1,1}(X,\IR)$ is $(1, b_2(X) - 3)$ (\cite[Cor. 23.11]{huy1}) the orthogonal complement $x^{\bot}$ in $H^{1,1}(X,\IR)$ is a negative definite subspace. Therefore $\lambda \notin x^{\bot}$, otherwise $(\lambda,\lambda) < 0$. We conclude $(x,\lambda) > 0$.
	
\end{proof}

\begin{proof}[Proof of Theorem \ref{main}] Let $f : X \rightarrow \IP^n$ a $\kdrei$--type fibration and set $L \coloneqq f^{\star} \ko_{\IP^n}(1)$. Then $\lambda \coloneqq c_1(L)$ is primitive and isotropic with respect to the Beauville--Bogomolov quadratic form by Lemma \ref{liso}. Let $d \coloneqq \Div(\lambda)$ denote the divisibility of $\lambda$. Consider the monodromy invariant $\vartheta : I_d(X) \rightarrow I_{n,d}$. By Lemma \ref{invariant} (iii) there exists an integer $b$ such that $\vartheta(\lambda)$ is also represented by $(H_{n,d}, (d,b))$. 

Then $\gcd(b,d) = 1$ by Remark \ref{bteilt} hence Lemma \ref{bm} gives a Beauville--Mukai system $\pi : X' \rightarrow \IP^n$ together with a primitive isotropic class $\alpha \in H^2(X', \IZ)$ such that $\Div(\alpha) = d$, $L' \coloneqq \pi^{\star} \ko_{\IP^n}(1)$ satisfies $c_1(L') = \alpha$ and $\vartheta(\alpha)$ is represented also by $(H_{n,d}, (d,b))$ \ie $\vartheta(\alpha) = \vartheta(\lambda)$.

Further by Lemma \ref{isopos} we have $(\omega, L) > 0$ and $(\omega', L') > 0$ for Kähler classes $\omega$ on $X$ and $\omega'$ on $X'$ as $L$ and $L'$ are isotropic and nef, therefore are contained in $\bar{\kk}_X \subset \bar{\kc}^{+}_X$ and $\bar{\kk}_{X'} \subset \bar{\kc}^{+}_{X'}$ respectively. Hence we can apply Lemma \ref{markcrit} to see that the pairs $(X, L)$ and $(X', L')$ are deformation equivalent in the sense of Definition \ref{pairdef}. By Proposition \ref{defo} there exist markings $\eta$ and $\eta'$ on $X$ and $X'$ respectively such that the pairs $(X,\eta)$ and $(X', \eta')$ are contained in the same connected component of the moduli of Lagrangian fibrations $\gU^{\circ}_{\lambda'}$ for a primitive isotropic class $\lambda' := \eta(c_1(L)) \in \Lambda$. By Theorem \ref{same} and Corollary \ref{prin} we have
$$ \pol(f) \ = \ \pol(\pi) \ = \ (1, \ldots ,1) $$ 
which concludes the proof. \end{proof}

\bibliographystyle{alpha}

\begin{thebibliography}{CvdG92}
	
	\bibitem[AK80]{compac}
	Allen Altman and Steven Kleiman.
	\newblock {Compactifying the Picard scheme}.
	\newblock {\em Adv. in Math.}, 35:50--112, 1980.
	
	\bibitem[Bea84]{Bea84}
	Arnaud Beauville.
	\newblock Vari\'{e}t\'{e}s k{\"a}hleriennes dont la premi\'{e}re classe de
	chern est nulle.
	\newblock {\em J. Differential Geom.}, 18:755--782, 1984.
	
	\bibitem[BL03]{BL}
	Christina Birkenhake and Herbert Lange.
	\newblock {\em Complex Abelian Varieties}, volume 302 Second Edition of {\em
		Grundlehren der mathematischen Wissenschaft}.
	\newblock Springer, 2003.
	
	\bibitem[Cam06]{camp05}
	Fr\'{e}d\'{e}ric Campana.
	\newblock Isotrivialit\'{e} de certaines familles k{\"a}hl\'{e}riennes de
	vari\'{e}t\'{e}s non projectives.
	\newblock {\em Mathematische Zeitschrift}, pages 147--156, 2006.
	
	\bibitem[CvdG92]{cili}
	Ciro Ciliberto and Gerard van~der Geer.
	\newblock {On the Jacobian of a hyperplane section of a surface}.
	\newblock {\em Classification of irregular varieties (Trento, 1990) Lecture
		Notes in Math., Springer, Berlin}, 1515:33--40, 1992.
	
	\bibitem[D'S79]{souza}
	Cyril D'Souza.
	\newblock {Compactification of generalized Jacobians}.
	\newblock {\em Proc. Indian Acad. Sci.}, 88:419--457, 1979.
	
	\bibitem[GHJ03]{huy1}
	Mark Gross, Daniel Huybrechts, and Dominic Joyce.
	\newblock {\em Calabi--Yau Manifolds and Related Geometries}.
	\newblock Springer, 2003.
	
	\bibitem[GL14]{LeGr13}
	Daniel Greb and Christian Lehn.
	\newblock Base manifolds for lagrangian fibrations on hyperk{\"a}hler
	manifolds.
	\newblock {\em Int. Math. Res. Notices}, 19:5483--5487, 2014.
	
	\bibitem[HL10]{huylehn}
	Daniel Huybrechts and Manfred Lehn.
	\newblock {\em The geometry of moduli spaces of sheaves}.
	\newblock Cambridge University Press, Cambridge, second edition, 2010.
	
	\bibitem[HO09]{HO09}
	Jun-Muk Hwang and Keiji Oguiso.
	\newblock Characteristic foliation on the discriminantal hypersurface of a
	holomorphic {L}agrangian fibration.
	\newblock {\em Amer. Journal of Math.}, 131:981--1007, 2009.
	
	\bibitem[Huy03]{huycone}
	Daniel Huybrechts.
	\newblock The {K}{\"a}hler cone of a compact hyperk{\"a}hler manifold.
	\newblock {\em Math. Ann}, pages 499--513, 2003.
	
	\bibitem[Hwa08]{Hwa08}
	Jun-Muk Hwang.
	\newblock Base manifolds for fibrations of projective irreducible symplectic
	manifolds.
	\newblock {\em Invent. Math. 174}, 3:625--644, 2008.
	
	\bibitem[Laz04]{larza}
	Robert Lazarsfeld.
	\newblock {\em {Positivity in Algebraic Geometry}}, volume I \& II of {\em
		{Ergebnisse der Mathematik, volumes 48 and 49}}.
	\newblock Springer, 2004.
	
	\bibitem[Mar10]{eyal3}
	Eyal Markman.
	\newblock Integral constraints on the monodromy group of the hyperk{\"a}hler
	resolution of a symmetric product of a {K}3 surface.
	\newblock {\em Internat. J. of Math. 21}, 21:169--223, 2010.
	
	\bibitem[Mar11]{eyal1}
	Eyal Markman.
	\newblock {A survey of Torelli and monodromy results for holomorphic-symplectic
		varieties}.
	\newblock In Wolfgang~Ebeling et. al., editor, {\em Complex and Differential
		Geometry}, volume~8, pages 257--323. Springer Proceedings in Math., 2011.
	
	\bibitem[Mar13]{eyalprime}
	Eyal Markman.
	\newblock Prime exceptional divisors on holomorphic symplectic varieties and
	monodromy reflections.
	\newblock {\em Kyoto J. Math.}, 53:345--403, No. 2, 2013.
	
	\bibitem[Mar14]{eyal2}
	Eyal Markman.
	\newblock {Lagrangian fibrations of holomorphic-symplectic varieties of
		K$3^{[n]}$-type}.
	\newblock In Anne Fr{\"u}hbis-Kr{\"u}ger et. al., editor, {\em Algebraic and
		Complex Geometry}, volume~71. Springer Proceedings in Math., 2014.
	
	\bibitem[Mat99]{Mat99}
	Daisuke Matsushita.
	\newblock On fibre space structures of a projective irreducible symplectic
	manifold.
	\newblock {\em Topology}, pages 38(1):79--83, 1999.
	
	\bibitem[Mat00]{Mat00}
	Daisuke Matsushita.
	\newblock Equidimensionality of {L}agrangian fibrations on holomorphic
	symplectic manifolds.
	\newblock {\em Math. Res. Lett.}, 7:389--391, 2000.
	
	\bibitem[Mat01]{Mat01}
	Daisuke Matsushita.
	\newblock Addendum to: On fibre space structures of a projective irreducible
	symplectic manifold.
	\newblock {\em Topology}, pages 38(1):79--83, 2001.
	
	\bibitem[Mat03]{Mat03}
	Daisuke Matsushita.
	\newblock Holomorphic symplectic manifolds and lagrangian fibrations.
	\newblock {\em Acta Appl. Math}, pages 75(1--3):117--123, 2003.
	
	\bibitem[Mat09]{Mat09}
	Daisuke Matsushita.
	\newblock On deformation of deformations of {L}agrangian fibrations.
	\newblock arXiv:0903.2098, 2009.
	
	\bibitem[Mat13]{matiso}
	Daisuke Matsushita.
	\newblock {On isotropic divisors on irreducible symplectic manifolds}.
	\newblock arXiv:1310.0896, 2013.
	
	\bibitem[Muk84]{mukais}
	Shigeru Mukai.
	\newblock {Symplectic structure of the moduli space of sheaves on an abelian or
		K3 surface}.
	\newblock {\em Invent. math}, 77:101--116, 1984.
	
	\bibitem[Muk87]{mukaibundle}
	Shigeru Mukai.
	\newblock {On the moduli space of bundles on K3 surfaces I}.
	\newblock {\em Tata Institute for fundamental research studies in mathematics},
	11:341--413, 1987.
	
	\bibitem[MY15]{eyalyoshi}
	Eyal Markman and Kota Yoshiokai.
	\newblock {A Proof of the Kawamata-Morrison Cone Conjecture for Holomorphic
		Symplectic Varieties of K$3^{[n]}$ or Generalized Kummer Deformation Type}.
	\newblock {\em Int Math Res Notices}, 24:13563--13574, 2015.
	
	\bibitem[Rie14]{ulrike1}
	Ulrike Riess.
	\newblock {On the Beauville conjecture}.
	\newblock arXiv:1409.3484v2, 2014.
	
	\bibitem[Saw03]{sawonab}
	Justin Sawon.
	\newblock Abelian fibred holomorphic symplectic manifolds.
	\newblock {\em Turkish Journal of Mathematics}, 27:197--230, 2003.
	
	\bibitem[Ver13]{Verb11}
	Misha Verbitsky.
	\newblock {Mapping class group and a global Torelli theorem for hyperk{\"a}hler
		manifolds}.
	\newblock {\em Duke Math. J.}, 162(15):2929--2986, 2013.
	
	\bibitem[Voi92]{Voi92}
	Claire Voisin.
	\newblock Sur la stabilit\'{e} des sous-vari\'{e}t\'{e}s lagrangiennes des
	vari\'{e}t\'{e}s symplectique holomorphes.
	\newblock {\em Complex projective geometry Cambridge Univ. Press},
	159:294--303, 1992.
	
	\bibitem[Voi02]{voih1}
	Claire Voisin.
	\newblock {\em Hodge Theory and Complex Algebraic Geometry I}, volume~76 of
	{\em Cambridge studies in advanced mathematics}.
	\newblock Cambridge University Press, 2002.
	
	\bibitem[Yos01]{yoshi}
	Kota Yoshioka.
	\newblock Moduli spaces of stable sheaves on abelian surfaces.
	\newblock {\em Math. Ann. 321}, 4:817--884, 2001.
	
	\bibitem[Yos12]{yoshibase}
	Kota Yoshioka.
	\newblock Bridgeland's stability and the positive cone of the moduli spaces of
	stable objects on an abelian surface.
	\newblock arXiv:1206.4838v2, 2012.
	
\end{thebibliography}

\end{document}